\documentclass[11pt]{elsarticle}
\usepackage{amsmath,amssymb,amsthm}
\usepackage{hyperref}
\usepackage{mathrsfs}
\usepackage{graphicx}
\usepackage{tikz}
\usepackage{enumerate}
\usepackage{float}
\usepackage{multirow}
\usepackage{etoolbox}
\usepackage{colortbl}    
\BeforeBeginEnvironment{enumerate}{\normalfont}

\theoremstyle{plain}
\newtheorem{theorem}{Theorem}[section]
\newtheorem{lemma}[theorem]{Lemma}

\newtheorem{proposition}[theorem]{Proposition}
\newtheorem{corollary}[theorem]{Corollary}

\theoremstyle{definition}
\newtheorem{definition}[theorem]{Definition}
\newtheorem{remark}[theorem]{Remark}
\newtheorem{example}[theorem]{Example}

\newcommand{\I}{\item}
\newcommand{\II}{\begin{enumerate}}
	\newcommand{\III}{\end{enumerate}}

\newcommand{\vsqu}[1][black]{
	\begin{tikzpicture}
		\draw[color=#1, rotate=45] (0,0) rectangle (0.2,0.2);
	\end{tikzpicture}
}
\newcommand{\dda}{\mathord{\mbox{\makebox[0pt][l]{\raisebox{-.4ex}{$\downarrow$}}$\downarrow$}}}

\newcommand{\ua}{\mathord{\uparrow}}

\newcommand{\rom}[1]{\rm{\uppercase\expandafter{\romannumeral #1}}}

\makeatletter
\def\ps@pprintTitle{%
	\let\@oddhead\@empty
	\let\@evenhead\@empty
	\def\@oddfoot{\reset@font\hfil\thepage\hfil}
	\let\@evenfoot\@oddfoot
}
\makeatother

\begin{document}
	
	\begin{frontmatter}
		
		\title{Co-Noetherian spaces}

		\author{Li Xiangrui}
		\ead{17731912783@163.com}
		\author{Li Qingguo\corref{a1}}
		\ead{liqingguoli@aliyun.com}
		\address{School of Mathematics, Hunan University, Changsha, Hunan, 410082, China}
		\cortext[a1]{Corresponding author.}
	\begin{abstract}
		
		In non-Hausdorff topology, many spaces exhibit significant separation properties, such as sober spaces, well-filtered spaces and d-spaces. These properties serve to fundamentally classify $T_0$ topological spaces. In this paper, we introduce and study a new class of topological spaces called co-Noetherian spaces, which can refine the classification of $T_0$ spaces. We discuss some basic properties of co-Noetherian spaces and obtain an equivalent characterization of compactness under the strong topology. Additionally, we investigate the connections among KC-spaces, strong R-spaces and co-Noetherian spaces. Moreover, we establish an equivalence between the category of $T_0$ co-Noetherian spaces with continuous mappings and a subcategory of the poset category. Finally, we provide counterexamples to show that the Hoare powerspace of a $T_0$ space may fail to be co-Noetherian, and that the Smyth powerspace of a co-Noetherian space need not be co-Noetherian.
\end{abstract}
		
	\end{frontmatter}
	\section{Introduction}
	
	In general topology, numerous separation axioms have been extensively studied, with particular emphasis on the properties of Hausdorff spaces. In recent years, there has been growing research interest in domain theory and non-Hausdorff topological spaces, focusing on sobriety, well-filteredness, d-spaces (also known as monotone convergence spaces), and several other separation axioms between $T_0$ and $T_2$.
	
	Sobriety is one of the most important and useful properties of non-Hausdorff topological spaces. It has been used in the characterizations of spectral spaces of commutative rings and the spaces which are determined by their open set lattices, see \cite{Hochster-1969}. With the development of domain theory, several significant classes of spaces have emerged from the study of sobriety: such as well-filtered spaces and d-spaces  (see \cite{abramsky1995domain,Gier-2003,Goubault-2013}). To better characterize spaces lying between $T_2$ and $T_0$, several concepts have been introduced: strong d-spaces in \cite{xu-zhao-2020}, R-spaces in \cite{xu-2016}, strongly well-filtered spaces in \cite{xu-2025}, and strong R-spaces in \cite{xu-yang-2025}. Extensive research has been conducted on these spaces.  Research on the d-completion, well-filterification and sobrification of $T_0$ spaces are constructed in \cite{Hoffmann-1979b,Keimel-2009,Liu-2020,Zhang-2017}. Discussions on certain hyperspaces of d-spaces, well-filtered spaces and sober spaces, etc. can be found in \cite{Heck-90,Heck-91,Heck-13,lyu-chen-jia-2022,xu-yang-2025}. The relationships among d-spaces, well-filtered spaces and sober spaces, etc. have been investigated in \cite{Xu-2024a,Xu-2024b,li-yuan-zhao-2020,Shen-2019,xu-2025}.
	
	
	In the study of the spaces mentioned above, considerable attention has been paid to compact saturated sets, and their investigation continues to be actively pursued in topology. For example, they are integral to the definitions of the patch topology and the Smyth powerspace (see \cite{Gier-2003,Goubault-2013,Johnstone}). Furthermore, Noetherian spaces form another important class of topological spaces that is closely related to compact saturated sets. The notion of Noetherianity originates from Noetherian rings, where every ideal is finitely generated (see \cite{Matsumura}). This ascending chain condition has applications in algebra, topology, and lattice theory (see \cite{Davey,Hartshorne,Matsumura}). In topology, this leads to the definition of a Noetherian space: a space whose lattice of open sets satisfies the ascending chain condition (see \cite{Hartshorne}).

	
	In fact, for Noetherian space, each open set can be regarded as a compact element in the lattice of open sets. 
	However, by duality, research on compact elements in the lattice of closed sets remains rather limited. For this reason, in this paper, we introduce a new class of topological spaces called co-Noetherian spaces, in which all closed sets are compact elements in the lattice of closed sets. The introduction of such a space allows for a more refined classification of $T_0$ spaces. We discuss some basic properties of co-Noetherian spaces and obtain an equivalent characterization of compactness under the strong topology. Additionally, we investigate the connections among KC-spaces, strong R-spaces and co-Noetherian spaces. Moreover, we establish an equivalence between the category of $T_0$ co-Noetherian spaces with continuous mappings and a subcategory of the poset category. Finally, we provide counterexamples to show that the Hoare powerspace of a $T_0$ space may fail to be co-Noetherian, and that the Smyth powerspace of a co-Noetherian space need not be co-Noetherian.


	
	
	\section{Preliminary}
	
	This section is devoted to a brief review of some basic concepts and notations to be used in the paper.

	Let $P$ be a poset. A nonempty subset $D$ of $P$ is \emph{directed} if every two points in $D$ have an upper bound in $D$. We write $D\subseteq^\uparrow P$ to indicate that $D$ is a $directed~subset$ of $P$. Similarly, we use $F\subseteq_{fin}P$ to indicate that $F$ is a finite subset of $P$. For any subset $A\subseteq P$, let $\uparrow A=\{x\in P:x\geq a~for~some~a\in A\}$ and $\downarrow A=\{x\in P:x\le a~for~some~a\in A\}$. Specifically, we write $\uparrow x=\uparrow\{x\}$ and $\downarrow x=\{x\}$. We call $A\subseteq P$ an \emph{upper set} (resp., \emph{lower set}) if $A=\uparrow A$ (resp., $A=\downarrow A$). A lower set $I$ is called a $principal~ideal$, if $I=\downarrow x$ for some $x\in P$.
	
	A poset $P$ is called a \emph{directed complete poset} (\emph{dcpo} for short), if $\bigvee D$ exists, for any $D\subseteq^\uparrow P$. A poset $P$ is called a \emph{semilattice }(resp., \emph{sup-semilattice}), if the infimum $\inf\{x,y\}=x\wedge y$ exists (respectively, the supremum $\sup\{x,y\}=x\vee y$ exists) for each pair $x,y\in P$. A poset is called a \emph{complete lattice} (resp., \emph{sup-complete poset}), if the supremun $\sup A=\bigvee A$ exists for each (resp., nonempty) subset $A$ of $P$.
	
	Let $P$ be a poset and $x,y\in P$. We say that $x$ is \emph{way-below}  $y$, in symbols $x\ll y$, iff for all directed subsets $D\subseteq L$ for which $\bigvee D$ exists, the relation $y\leqslant \bigvee D$ always implies the existence of a $d\in D$ with $x\leqslant d$. In particular, $x$ is called a \emph{compact element}, if $x\ll x$.
	
	For each $x\in P$, we denote by $\dda x$ the set of all elements are \emph{way below} $x$. A dcpo $P$ is called a \emph{continuous dcpo} (\emph{domain} for short) if it satisfies the $axiom$ of $approximation$:
	\begin{center}
		\emph{(${\forall} x\in P$)}$x=\bigvee^{\ua}\dda x$,
	\end{center}
	i.e. for all $x\in L$, the set $\dda x=\{u\in L\mid u\ll x\}$ is directed and $x=\bigvee \{u\in L\mid u\ll x\}$.
	
	A poset is called \emph{Noetherian}, if it satisfies the \emph{ascending chain condition}: every ascending chain has a greatest member. For a Noetherian poset, we have the following remarks. It implies that all of Noetherian posets are domains.
	
	\begin{remark}
		For a poset $P$, the following conditions are equivalent:
		\II
		\I[(1)] $P$ is Noetherian.
		\I[(2)] $P$ is a dcpo and all elements of $P$ are compact.
		\III
	\end{remark}
	
	See \cite{Gier-2003}, there are some topologies on the poset $P$. The \emph{upper topology} on the poset $P$ is generated by the complements of $\{\downarrow x:x\in P\}$ and denote it by $\upsilon(P)$. A subset $U$ of $P$ is \emph{Scott open} if $U=\uparrow U$ and any directed set $D$ for which $\bigvee D$ exists, $\bigvee D\in U$ implies $D\cap U\not=\emptyset$. The set of all Scott open sets of $P$ forms the \emph{Scott topology} on $P$, which is denoted by $\sigma(P)$. The set of all upper sets of $P$ is the \emph{Alexandroff topology} on $P$ and is denoted by $\alpha(P)$.

	For a topological space $(X,\tau)$, a subset $A$ of $P$ is called \emph{saturated} if $A=\bigcap\{U\in\tau:A\subseteq U\}$. A nonempty subset $C$ of $X$ is called $irreducible$ if for any closed sets $C_1,C_2$, $C\subseteq C_1\cup C_2$ implies $C\subseteq C_1$ or $C\subseteq C_2$. For a subset $A\subseteq X$, we will use $cl(A)$ to denote the closure of $A$. Specifically, $cl(x)=cl(\{x\})$ for any $x\in X$. Note that $A$ is irreducible if and only if $cl(A)$ is irreducible.
	
	For a $T_0$ space $(X,\tau)$, the partial order $\le_\tau$, defined by $x\le_\tau y$ if and only if $x\in cl(y)$. In the following, for a $T_0$ space $(X,\tau)$, the order always refers to the specialization order if there is no other explanation. In the following, we shall use $\le$ to denote the specialization order $\le_\tau$.

	In this paper, for a topological space $(X,\tau)$, there are some notations: 
	
	(1) $\mathcal{S}(X,\tau)$, the set of all saturated sets;
	
	(2) $\mathcal{Q}(X,\tau)$, the set of all compact saturated sets;
	
	(3) $\mathcal{D}(X,\tau)$, the set of all directed sets of $(X,\le_{\tau})$;
	
	(4) $\Gamma(X,\tau)$, the set of all closed sets;
	
	(5) $Irr_c(X,\tau)$, the set of all irreducible closed set.
	
	We will simplify the notation by omitting the topology $\tau$, when the topology is unambiguous. It means that we write $\mathcal{S}(X)$ for $\mathcal{S}(X,\tau)$, $\mathcal{Q}(X)$ for $\mathcal{Q}(X,\tau)$.
		
	A $T_0$ space $(X,\tau)$ is a \emph{d-space} if $(X,\le)$ is a dcpo and $\tau\subseteq\sigma(X,\le)$.
	
	The following result is well-known.
	
	\begin{proposition}
		Let $X$ be a d-space and $C$ be a nonempty closed set. Then $A=\downarrow max(A)$ with respect to the specialization order.
	\end{proposition}
	
	\begin{proposition}\rm{\cite{li-yuan-zhao-2020}}
		For a $T_0$ space $(X,\tau)$, the following conditions are equivalent:
		\II
		\I[(1)] $X$ is a d-space.
		\I[(2)] For any $D\in\mathcal{D}$ and $U\in\tau$, $\bigcap\{\uparrow d:d\in D\}\subseteq U$ implies $\uparrow d\subseteq U$.
		\I[(3)] For any filtered family $\{\uparrow F_i:F_i\subseteq_{fin}X,i\in I\}$ and any $U\in\tau$, $\bigcap\{\uparrow F_i:i\in I\}\subseteq U$ implies $\uparrow F_i\subseteq U$ for some $i\in I$.
		\III
	\end{proposition}
	
	\begin{definition}\cite{shen-xi-xu-zhao-2020}
		Let $(X,\tau)$ be a $T_0$ space. 
		\II
		\I[(1)] For any $U,V\in\tau$, we define $U\ll V$ if and only if each open cover of $V$ has a fnite subfamily that covers $U$.
		\I[(2)] $(X,\tau)$ is called \emph{core-compact} if, for each $U\in\tau$, it holds that $$U=\bigcup\{V\in\tau:V\ll U\}.$$
		\I[(3)] A subfamily $\mathcal{F}\subseteq\tau$ is called a \emph{$\ll$-filtered family} if for any $U_1,U_2\in\mathcal{F}$, there exists $U_3\in\mathcal{F}$ such that $U_3\ll U_1$ and $U_3\ll U_2$.
		\I[(4)] $(X,\tau)$ is called \emph{open well-fltered} if for each $\ll$-filtered family $\mathcal{F}\subseteq \tau$ and $U\in\tau$,
		$$\bigcap\mathcal{F}\subseteq U~\Rightarrow~ V\subseteq U~for~some~V\in\mathcal{F}.$$
		\III
	\end{definition}

	\begin{definition}(\cite{Alas-2002,Goubault-2013,xu-2025,xu-yang-2025,xu-zhao-2020}) Let $(X,\tau)$ be a $T_0$ space, then
		\II
		\I[(1)] $(X,\tau)$ is a \emph{c-space} if, for each $x\in X$ and $U\in N^o(x)$, there is a $y\in U$ such that $x\in int(\uparrow y)$.
		
		\I[(2)] $(X,\tau)$ is \emph{sober} if every irreducible closed subset of $X$ is the
		closure of a unique point.
		
		\I[(3)] $(X,\tau)$ is \emph{well-filtered} if for each filtered family $\mathcal{C}\subseteq\mathcal{Q}(X)$ and $U\in\tau$, $\bigcap\mathcal{C}\subseteq U$ implies $K\subseteq U$ for some $K\in\mathcal{C}$.
		
		\I[(4)] $(X,\tau)$ is \emph{strongly well-filtered} if for each filtered family $\mathcal{K}\subseteq\mathcal{Q}(X)$, $K_0\in\mathcal{Q}(X)$ and $U\in\tau$, $K_0\cap\bigcap\mathcal{K}\subseteq U$ implies $K\in\mathcal{K}$ for some $K_0\cap K\subseteq U$.
		
		\I[(5)] $(X,\tau)$ is a \emph{strong d-space} if for each $x\in X$, $U\in\tau$ and $D\in\mathcal{D}(X)$, $\uparrow x\cap\bigcap\{\uparrow d:d\in D\}\subseteq U$ implies $\uparrow x\cap\uparrow d\subseteq U$ for some $d\in D$.
		
		\I[(6)] $(X,\tau)$ is an \emph{R-space} if for each family $\mathcal{F}=\{\uparrow F_i:F\subseteq_{fin}X,i\in I\}$ and $U\in\tau$, $\bigcap\mathcal{F}\subseteq U$ implies $\bigcap\{\uparrow{F_i}:i\in A\}\subseteq U$ for some $A\subseteq_{fin}I$.
		
		\I[(7)] $(X,\tau)$ is a \emph{strong R-space} if for each family $\mathcal{K}\subseteq\mathcal{Q}(X)$ and $U\in\tau$, $\bigcap\mathcal{K}\subseteq U$ implies that $\bigcap\mathcal{F}\subseteq U$ for some finite subfamily $\mathcal{F}\subseteq\mathcal{K}$.
		\III
		\end{definition}
		
		\begin{remark}\label{s-r}
			Let $(X,\tau)$ be a strong R-space. Then $(X,\tau)$ is an R-space and a strong well-filtered space.
		\end{remark}
		
		\begin{proposition}\rm{\cite{zhao-ho-2015}}\label{sup-sob}
			For a sup-complete poset $P$, $(P,\upsilon(P))$ is sober.
		\end{proposition}
		
		\begin{proposition}\rm{\cite{xu-yang-2025}}\label{sup-coh}
			For a sup-complete poset $P$, $(P,\upsilon(P))$ is coherent.
		\end{proposition}

		\begin{definition}\cite{Gier-2003,Keimel-2009}
			Let $(X,\tau)$ be a topological space.
			\II
			\I[(1)] $\tau^s$ is called the strong topology of $\tau$, if $\tau^s$ has a subbase that $\tau\cup \Gamma(X)$. 
			\I[(2)] $\tau^p$ is called the patch topology of $\tau$, if $\tau^p$ has a subbase that $\tau\cup\{X-K:K\in \mathcal{Q}(X)\}$.
			\III 
		\end{definition}
		
		A topological $(X,\tau)$ is called a KC-space if all compact sets are closed. The strong topology is also called the Skula topology and KC-spaces called the $T_B$ spaces. In this paper, we use the notion of strong topology and KC-spaces (see \cite{skula-1969,Wil-1967}). For KC-spaces, we have the following properties. There are some propositions about KC-spaces.
		
		\begin{proposition}
			Let $(X,\tau)$ be a topological space. Consider the following conditions:
			\II
			\I[(1)] $(X,\tau)$ is a KC-space.
			\I[(2)] $\mathcal{Q}(X)\subseteq \Gamma(X)$.
			\I[(3)] $\tau=\tau^p$.
			\III
			Then $(1)\Rightarrow (2) \Leftrightarrow (3)$, and all the conditions are equivalent if $(X,\tau)$ is a $T_0$ space.
		\end{proposition}
		\begin{proof}
			Clearly, $(1)\Rightarrow (2)\Leftrightarrow (3)$, by the definitions.
			
			 For $(2)\Rightarrow (1)$, we only need to prove that $(X,\tau)$ is a $T_1$ space; hence, all compact sets are saturated. 
			 
			 Given $x\in X$. Then $\uparrow x\in\mathcal{Q}(X)$. Assume that $y\in \uparrow x$. Then $\uparrow y\in\mathcal{Q}(X)\subseteq\Gamma(X)$ and $x\le y$ ($x\in cl(y)$) implies that $x\in \uparrow y$ since $\uparrow y$ is closed. Due to the fact that $(X,\tau)$ is $T_0$, we know that $y\in\uparrow x$ and $x\in\uparrow y$ implies that $x=y$. Thus $(X,\tau)$ is a $T_1$ space. 
		\end{proof}
			
		Recall that a topological space is called \emph{coherent} if $A,B\in\mathcal{Q}(X)$ always implies that $A\cap B\in\mathcal{Q}(X)$.
		
		\begin{remark}
			All $T_2$ spaces are KC-space. And all KC-spaces are coherent $T_1$ space.
		\end{remark}

		Let $(X,\tau)$ be a $T_0$ space and $\Gamma^*(X)$ be the set of all nonempty closed sets of $X$. For $U\in\tau$, we denoted by $\vsqu U=\{C\in\Gamma^*(X):C\cap U\not=\emptyset\}$. The Hoare power space of $(X,\tau)$ is $\Gamma^*(X)$ endowed with topology generated by $\{\vsqu U:U\in\tau\}$ and is denoted by $P_H(X)$.  Similarly, let $\mathcal{Q}^*(X)$ be the set of all nonempty compact saturated sets of $X$. For $U\in\tau$, let $\square U=\{K\in\mathcal{Q}^*(X):K\subseteq U\}$. The \emph{Smyth power space} of $(X,\tau)$ is $\mathcal{Q}^*(X)$ endowed with the topology generated by $\{\square U:U\in\tau\}$ and is denoted by $P_S(X)$. In this paper, for each family $\mathcal{F}$ of nonempty compact saturated set, we define $\downarrow_{P_S(X)}\mathcal{F}=\{Q\in\mathcal{Q}^*(X):K\subseteq Q~for~some~K\in\mathcal{F}\}$, and $\uparrow_{P_S(X)}\mathcal{F}=\{Q\in\mathcal{Q}^*(X):Q\subseteq K~for~some~K\in\mathcal{F}\}$.


		The following results regarding Smyth power spaces are well-known.
		
		\begin{theorem}\label{lyu} \rm{\cite{lyu-chen-jia-2022}}
			Let $X$ be a topological space. The following statements are equivalent:
			\II
			\I[(1)] $X$ is locally compact.
			\I[(2)] $P_S(X)$ is a c-space.
			\I[(3)] $P_S(X)$ is locally compact.
			\I[(4)] $P_S(X)$ is core-compact.
			\III
		\end{theorem}
		
		\begin{theorem}\label{xu}\rm{\cite{xu-shen-zhao-2020}}
			For a $T_0$ space $X$, the following conditions are equivalent:
			\II
			\I[(1)] $X$ is well-filtered.
			\I[(2)] $P_S(X)$ is a d-space.
			\I[(3)] $P_S(X)$ is well-filtered.
			\III
		\end{theorem}

\section{KC-spaces and strong R-spaces}

In this section, we discuss the relationship between KC-spaces and strong R-spaces and investigate some properties of strong R-spaces and some compactness properties related to the patch topology. The relationship between KC-spaces and strong R-spaces is established as follows.

\begin{proposition}\label{KC-s-d}
	If $(X,\tau)$ is a KC-space, then $(X,\tau)$ is a strong R-space. 
\end{proposition}
\begin{proof}
	Assume $U\in\tau$ and $\mathcal{K}\subseteq\mathcal{Q}(X)$ such that $\bigcap\mathcal{F}\not\subseteq U$ for each $\mathcal{F}\subseteq_{fin}\mathcal{K}$.

	Fix $K_0\in\mathcal{K}$. Then $C=(X-U)\cap K_0$ is compact. By assumption, $C$ is a nonempty closed set. For each $K\in\mathcal{K}$, $C\cap K$ is also compact and it is therefore closed. Furthermore, $\mathcal{A}=\{\bigcap\mathcal{F}\cap C:\mathcal{F}\subseteq_{fin}\mathcal{K}\}$ is a filtered family of nonempty closed sets of $(C,\tau|_C)$. It follows that $\bigcap\mathcal{A}\not=\emptyset$. That is $\bigcap\mathcal{K}\not\subseteq U$. Thus $(X,\tau)$ is a strong R-space.
\end{proof}

While all KC-spaces are strong R-spaces, the following example shows that a KC-space need not be sober. Conversely, the Sierpi\'{n}ski space is sober but not a KC-space.

\begin{example}
	Let $\mathbb{R}$ be the set of real numbers with the co-countable topology $\tau$. Then $(\mathbb{R},\tau)$ is a KC-space, but not sober.

	Obviously, $\mathcal{Q}(\mathbb{R})=\mathcal{Q}_f(\mathbb{R})=\{F:F\subseteq_{fin}\mathbb{R}\}$. It follows that $(\mathbb{R},\tau)$ is a KC-space.
	
	Given an infinite subset $A$ of $\mathbb{R}$. Choose a countable subset $A_0=\{a_n:n\subseteq\mathbb{N}\}$ of $A$. Then $$A\subseteq \mathbb{R}=\bigcup\{(X-A_0)\cup\{a_n\}:n\in\mathbb{N}\}.$$
	Since $(X-A_0)\cup\{a_n\}$ is open for each $n\in\mathbb{N}$, $A$ is not compact. Hence the equation holds.
	
	For any nonempty sets $U_1,U_2\in\tau$, we have $U_1\cap U_2\not=\emptyset$. So $\mathbb{R}$ is irreducible. But $cl(\mathbb{R})=\mathbb{R}\not=cl(\{x\})$ for any $x\in\mathbb{R}$. Thus $(\mathbb{R},\tau)$ is not sober.
\end{example}

\begin{proposition}
	Let $(X,\tau)$ be a $T_1$ coherent well-filtered space. If $(X,\tau)$ is locally compact, then $(X,\tau)$ is a KC-space.
\end{proposition}
\begin{proof}
	Given $K\in\mathcal{Q}(X)$ and $y\not\in K$. We only need to prove that $y\not\in cl(K)$. Suppose not, then $U\cap K\not=\emptyset$ for each $U\in N^o(y)$. Indeed, $U\cap K$ is a infinite set (otherwise, $U-(U\cap K)\in N^o(y)$, a contradiction). By local compactness, we obtain that there exists $Q_U\in\mathcal{Q}(X)$ such that $y\in \int(Q_U)\subseteq Q\subseteq U$ for each $U\in N^o(y)$. It implies that $K\cap Q_U$ is infinite and compact by coherence of $X$. Furthermore, $$\bigcap (K\cap Q_U)\subseteq K\cap\{y\}=\emptyset.$$
	Since $X$ is well-filtered, we have $K\cap Q_U=\emptyset$ for some $U\in N^o(y)$, a contradiction.
\end{proof}

 We now proceed to establish the connection among strong R-spaces, strongly well-filtered spaces and R-spaces. It is clear that all strong R-spaces are strongly well-filtered R-spaces. Conversely, we have the following proposition.
 
 Before proceeding, recall that a $T_0$ space $(X,\tau)$ is called locally hypercompact, if for each $x\in X$ and $U\in N^o(x)$, there is a finite set $F\subseteq X$ such that $x\in int(\uparrow F)$ and $F\subseteq U$. 

\begin{proposition}\label{wf-s-r} 
	Let $(X,\tau)$ be a $T_0$ space.
	\II
	\I[(i)] If $(X,\tau)$ is a coherent well-filtered space, then $(X,\tau)$ is a strong R-space.
	\I[(ii)] If $(X,\tau)$ is a locally hypercompact R-space, then $(X,\tau)$ is a strong R-space.
	\III
\end{proposition}
\begin{proof}
	Assume $U\in\tau$ and $\mathcal{K}\subseteq\mathcal{Q}(X)$ with $\bigcap\mathcal{K}\subseteq U$. It is sufficient to prove that $\bigcap\mathcal{F}\subseteq U$ for some $\mathcal{F}\subseteq_{fin}\mathcal{K}$.
	
	(i) By coherence of $(X,\tau)$, $\{\bigcap\mathcal{F}:\mathcal{F}\subseteq_{fin}\mathcal{K}\}$ is a filter of compact saturated sets. Since $(X,\tau)$ is well-filtered, there is an $\mathcal{F}\subseteq_{fin}\mathcal{K}$ such that $\bigcap\mathcal{F}\subseteq \mathcal{U}$.
	
	(ii) Given a compact saturated set $K\in\mathcal{Q}(X)$. Then $K=\bigcap\{V\in\tau:K\subseteq U\}$. Let $U$ be an open set with $K\subseteq V$. For each $x\in K$, there is a finite set $F_x$ such that $x\in int(\uparrow F_x)$ and $F_x\subseteq V$, because $(X,\tau)$ is locally hypercompact. Due to the fact that $K$ is compact, it exists a finite set $F_{K,V}$ satisfying $K\subseteq int(\uparrow F_{K,V})$ and $F_{K,V}\subseteq U$. It holds that
	$$\bigcap\mathcal{K}=\bigcap_{K\in\mathcal{K}}\bigcap_{K\subseteq V,V\in\tau}\uparrow F_{K,V}\subseteq U.$$
	There is a finite family $\{F_{K_i,V_i}:i=1,2,...,n\}$ satisfying $\bigcap_{i=1}^n\uparrow F_{K_i,V_i}\subseteq U$, where $K_i\in\mathcal{K},K_i\subseteq V_i\in \tau$, since $(X,\tau)$ is an R-space. It follows that $\bigcap_{i=1}^n K_i\subseteq \bigcap_{i=1}^n\uparrow F_{K_i,V_i}\subseteq U.$
\end{proof}

From Remark \ref{s-r}, every strong R-space is a strongly well-filtered R-space. However, an infinite set endowed with the co-finite topology is $T_1$(thus an R-space), but fails to be well-filtered; consequently, it is not strongly well-filtered. In what follows, we provide an example to demonstrate that not all strongly well-filtered spaces are R-spaces.

\begin{example}
	Let $P=\{n_1:n\in\mathbb{N}\}\cup\{n_2:n\in\mathbb{N}\}$. 
	
	Define an order on $P$ as follows (see Figure 1):
	
	(i) $n_1\le m_2$ for $n\le m$;
	
	(ii) $x\le x$ for all $x\in P$.
	
	\begin{figure}[h]
		\centering
		\begin{tikzpicture}[scale=0.5]
			\path (-8,0) node[left]{} coordinate (a);
			\fill (a) circle (2pt);
			\path (-4,0) node[left]{} coordinate (b);
			\fill (b) circle (2pt);
			
			\path (0,0) node[left]{} coordinate (c);
			\fill (c) circle (2pt);
			\path (4,0) node[left]{} coordinate (d);
			\fill (d) circle (2pt);
			
			\path (-8,3) node[left]{} coordinate (e);
			\fill (e) circle (2pt);
			\path (-4,3) node[left]{} coordinate (f);
			\fill (f) circle (2pt);
			
			\path (0,3) node[left]{} coordinate (g);
			\fill (g) circle (2pt);
			\path (4,3) node[left]{} coordinate (h);
			\fill (h) circle (2pt);
			
			\draw (a) -- (e);			
			\draw (a) -- (f);
			\draw (a) -- (g);
			\draw (a) -- (h);
			
			\draw (b) -- (f);
			\draw (b) -- (g);
			\draw (b) -- (h);
			
			\draw (c) -- (g);
			\draw (c) -- (h);
			
			\draw (d) -- (h);			
			\draw[style=dashed] (6,0) -- (9,0);
			\draw[style=dashed] (6,3) -- (9,3);			
			
			\path (-7.1,-1) node[left]{$1_1$} coordinate;
			\path (-3.1,-1) node[left]{$2_1$} coordinate;
			\path (0.9,-1) node[left]{$3_1$} coordinate;
			\path (4.9,-1) node[left]{$4_1$} coordinate;
			
			\path (-7.1,4) node[left]{$1_2$} coordinate;
			\path (-3.1,4) node[left]{$2_2$} coordinate;
			\path (0.9,4) node[left]{$3_2$} coordinate;
			\path (4.9,4) node[left]{$4_2$} coordinate;

		\end{tikzpicture}
		\caption{$(P,\le)$ }\label{fig1}
	\end{figure}
	Let $P$ be endowed with Scott topology. If $\mathcal{F}$ is a filter base consisting of nonempty compact saturated sets, then $|\mathcal{F}|$ is finite. It implies $(P,\sigma(P))$ is strongly well-filtered. However, $\bigcap \{\uparrow n_1:n\in\mathbb{N}\}=\emptyset$ and $\bigcap\{\uparrow n_1:n\in F\}\not=\emptyset$ for each $F\subseteq_{fin}\mathbb{N}$. Thus $(P,\sigma(P))$ is not an R-space.
\end{example}

Next, we discuss some properties of the Smyth powerspace concerning strong R-spaces.

\begin{theorem}
	Let $(X,\tau)$ be a $T_0$ space, consider the following condition:
	
	\II
	\I[(1)] $P_S(X)$ is a strong R-space.
	\I[(2)] $P_S(X)$ is an R-space.
	\I[(3)] $X$ is a strong R-space.
	\III
	
	Then $(1)\Rightarrow(2)\Rightarrow(3)$. If $(X,\tau)$ is locally compact, then $(1)\Leftrightarrow(2)$. If $(X,\tau)$ is coherent, then $(2)\Leftrightarrow(3)$. 
\end{theorem}
\begin{proof}
	$(1)\Rightarrow (2)$: It is obvious by Remark \ref{s-r}.
	
	$(2)\Rightarrow (3)$: Given a family $\mathcal{K}$ of compact saturated sets and an open set $U\in\tau$ satisfying $\bigcap\mathcal{K}\subseteq U$. Then $\bigcap\{\uparrow_{P_S(X)}K:K\in\mathcal{K}\}\subseteq\square U$. It implies that $\bigcap\{\uparrow_{P_S(X)}K:K\in\mathcal{F}\}\subseteq\square U$ for some finite subfamily $\mathcal{F}\subseteq\mathcal{K}$. That is $\bigcap\mathcal{F}\subseteq U$.
	
	Let $(X,\tau)$ be a locally compact space.
	
	$(2)\Rightarrow (1)$: By Theorem \ref{lyu}, $P_S(X)$ is a c-space. Thus, $P_S(X)$ is a locally hypercompact space. We conclude that $P_S(X)$ is a strong R-space by Proposition \ref{wf-s-r}.
	
	Let $(X,\tau)$ be a coherent space.
	
	$(3)\Rightarrow (2)$: Let $\mathcal{K}$ be a family of compact saturated sets and $\mathcal{U}$ be an open set of $P_S(X)$ with $\bigcap\{\uparrow_{P_S(X)}K:K\in\mathcal{K}\}\subseteq\mathcal{U}$. Due to coherence of $(X,\tau)$, $\bigcap\{\uparrow_{P_S(X)}K:K\in\mathcal{F}\}=\uparrow_{P_S(X)}\bigcap\mathcal{F}$, for each $\mathcal{F}\subseteq_{fin}\mathcal{K}$. Then $$\bigcap\{\uparrow_{P_S(X)}\mathcal{F}:\mathcal{F}_{fin}\mathcal{K}\}=\bigcap\{\uparrow_{P_S(X)}K:K\in\mathcal{K}\}\subseteq\mathcal{U}.$$ By Theorem \ref{xu}, there exists $\mathcal{F}_0\subseteq_{fin}\mathcal{K}$ such that $\uparrow_{P_S(X)}\bigcap\mathcal{F}\subseteq \mathcal{U}$. That is, $\bigcap\{\uparrow_{P_S(X)}K:K\in\mathcal{F}_0\}\subseteq\mathcal{U}$. 
\end{proof}

For \cite[Question 7.4.]{xu-2025}, we show that the Smyth powerspace of a locally compact coherent well-filtered space is a strong R-space; hence, it is strongly well-filtered, which partially settles the problem.

Finally, we discuss some compactness properties of the patch topology.

\begin{proposition}\rm{\cite{Goubault-2013,xi-2017}}
	Let $(X,\tau)$ be a $T_0$ space and $\tau^p$ be the patch topology of $(X,\tau)$. Then $(X,\tau^p)$ is compact if and only if $(X,\tau)$ is compact, coherent and well-filtered.
\end{proposition}

By Propositions \ref{wf-s-r} and \ref{KC-s-d}, we have the following corollaries.

\begin{corollary}
	Let $(X,\tau)$ be a $T_0$ space and $\tau^p$ be the patch topology of $(X,\tau)$. Then $(X,\tau^p)$ is compact if and only if $(X,\tau)$ is a compact coherent strong R-space. In addition, for each compact KC-space $(Y,\nu)$, $(Y,\nu^p)$ is compact. 
\end{corollary}

\section{Co-Noetherian spaces}

In this section, we introduce a new class of topological spaces, it is defined by modifying well-filteredness through replacing compact saturated sets with open sets in the definition. Now we supply the following definition.

\begin{definition}
	Let $(X,\tau)$ be a topological space. $(X,\tau)$ is called a co-Noetherian space, if for each open set $V$ and each filter basis $\mathcal{U}$ of open sets with $\bigcap\mathcal{U}\subseteq V$, there is a $U\in\mathcal{U}$ with $U\subseteq V$.
\end{definition}

\begin{proposition}\label{co-noe}
	Let $(X,\tau)$ be a topological space. Then the following conditions are equivalent:
	\II
	\I[(1)] $X$ is a co-Noetherian space.
	\I[(2)] For each open set $V$ and each family $\mathcal{U}$ of open sets with $\bigcap\mathcal{U}\subseteq V$, there is a $\mathcal{U}_0\subseteq_{fin}\mathcal{U}$ with $\bigcap\mathcal{U}_0\subseteq V$.
	\I[(3)] For each open set $V$ and each filter basis $\mathcal{S}$ of saturated sets with $\bigcap\mathcal{S}\subseteq V$, there is an $S\in\mathcal{S}$ with $S\subseteq V$.
	\I[(4)] For each open set $V$ and each family $\mathcal{S}$ of saturated sets with $\bigcap\mathcal{S}\subseteq V$, there is a $\mathcal{S}_0\subseteq_{fin}\mathcal{S}$ with $\bigcap\mathcal{S}_0\subseteq V$.
	\I[(5)] $\Gamma(X)=\{cl(F):F\subseteq_{fin}X\}$.
	\III
\end{proposition}

\begin{proof}
	It is clear that $(4)\Rightarrow (3) \Rightarrow (1)$ and $(4)\Rightarrow (2) \Rightarrow (1)$. 
	
	Now we only need to prove that $(1)\Rightarrow (5) \Rightarrow (4)$.
	
	$(1)\Rightarrow (5)$: Given a closed set $C$. Then $$X-\bigcup\{cl(F):F\subseteq_{fin}C\}=\bigcap\{X-cl(F):F\subseteq_{fin}C\}=X-C.$$
	
	It is clear that $X-C\in\tau$ and $\{X_cl(F):F\subseteq_{fin}C\}$ is a basis of open sets. There is a finite set $F\subseteq_{fin}C$, such that $X-cl(F)\subseteq X-C$ (i.e., $C\subseteq cl(F)$). Thus $C=cl(F)$.
	
	$(5)\Rightarrow (4)$: Assume $U\in \tau$ and $\{S_i:i\in I\}\subseteq S(X)$ such that $\bigcap_{i\in I}S_i\subseteq U$. 
	
	That is, $X-U\subseteq \bigcup_{i\in I}(X-S_i)$. Then there is a finite set $F\subseteq_{fin}X-U$ with $X-U=cl(F)$.
	
	Fix $a\in F$. Then $a\in cl(F)\subseteq \bigcup_{i\in I}X-S_i$. It implies that there exists an $i_a\in I$ such that $a\not\in S_{i_a}$. It follows that $cl(\{a\})\cap S_{i_a}=\emptyset$. It implies that $\bigcup\{cl(\{a\}):a\in F\}\cap\bigcap \{S_{i_a}:a\in F\}=\emptyset$. That is $cl(F)\subseteq \bigcup\{X-S_{i_a}:a\in F\}$. Thus $\bigcap\{S_{i_a}:a\in F\}\subseteq U$.
\end{proof}

\begin{remark}
	Recall that a space $(X,\tau)$ is called Noetherian, if all open sets are compact. That is, for each $\mathcal{U}\subseteq\tau$ and $V\in\tau$, $V\subseteq\bigcup\mathcal{U}$ always implies that $V\subseteq\bigcup\mathcal{U}_0$ for some $\mathcal{U}_0\subseteq_{fin}\mathcal{U}$.
	
	The reason that the space is called co-Noetherian spaces is that it satisfies the following property, for each $\mathcal{E}\subseteq\Gamma(X)$ and $C\in\Gamma(X)$, $C\subseteq\bigcup\mathcal{E}$ always implies that $C\subseteq\bigcup\mathcal{E}_0$ for some $\mathcal{E}_0\subseteq_{fin}\mathcal{E}$.
\end{remark}

\begin{proposition}
	Let $(X,\tau)$ be a co-Noetherian $T_0$ space. Then $(X,\tau)$ is a strong R-space.
\end{proposition}
\begin{proof}
	It is straightforward, by $\mathcal{Q}(X)\subseteq\mathcal{S}(X)$ and Proposition \ref{co-noe}.
\end{proof}

\begin{corollary}\label{co-noe-sober}
	All $T_0$ co-Noetherian spaces are sober.
\end{corollary}
\begin{proof}
	Let $(X,\tau)$ be a co-Noetherian space. Given an irreducible closed set $C$ of $(X,\tau)$. Then $C=cl(F)$ for some $F\subseteq_{fin}C$. Since $C=cl(F)=\bigcup\{cl(\{x\}):x\in F\}$, we have $C=cl(\{x\})$ for some $x\in F$. Thus $(X,\tau)$ is sober.
\end{proof}

See \cite[Exercise III-3.21]{Gier-2003}, any dcpo that has no infinite anti-chain is quasi-continuous. There is a similar property.

\begin{corollary}
	Any dcpo that has no infinite anti-chain endowed with Scott topology is a co-Noetherian space.
\end{corollary}
\begin{proof}
	Let $L$ be a dcpo and $C$ be a Scott closed set of $L$. Then $C=\downarrow\max(C)=cl(\max(C))$. By hypothesis, $\max(C)$ is finite. Thus $L$ endowed with Scott topology is a co-Noetherian space.
\end{proof}

Then we obtain some equivalent characterizations for the compactness of the strong topology as follows.

\begin{theorem}\label{str-compact-1}
	Let $(X,\tau)$ be a topological space. Then $X$ endowed with strong topology $\tau^s$ of $(X,\tau)$ is compact if and only if $(X,\tau)$ is a Noetherian and a co-Noetherian space.
\end{theorem}

\begin{proof}
	Suppose $X$ endowed with strong topology $\tau^s$ is compact. 
	
	Then all closed sets of $(X,\tau^s)$ are compact in $(X,\tau^s)$. Hence, all open sets $(X,\tau)$ are compact in $(X,\tau^s)$. It follows that $(X,\tau)$ is Noetherian.  
	
	Given a closed set $C$ of $(X,\tau)$. Then it is a compact closed set in $(X,\tau^s)$. Clearly, $\{cl(\{x\}):x\in C\}$ is an open cover of $C$ in $(X,\tau^s)$. There is a finite set $F\subseteq C$ such that $C=\bigcap\{cl(\{x\}):x\in F\}=cl(F)$. Thus $X$ is co-Noetherian by Proposition \ref{co-noe}. 
	
	Conversely, suppose $(X,\tau)$ is a Noetherian and a co-Noetherian space. Let $\{U_i\in\tau:i\in I\}\cup\{C_j\in\Gamma(X,\tau):j\in J\}$ be an open cover of $X$ in $(X,\tau^s)$. Since $(X,\tau)$ is a Noetherian space, there is a finite set $F_1\subseteq I$ such that $\bigcup\{U_i:i\in I\}= \bigcup\{U_i:i\in F_1\}$. Hence 
	$$X-\bigcup\{C_j:j\in J\}=\bigcap\{X-C_j:j\in J\}\subseteq \bigcup\{U_i:i\in F_1\}.$$ 
	Because $(X,\tau)$ is a co-Noetherian space, there is a finite set $F_2\subseteq J$ such that $\bigcap\{X-C_j:j\in F_2\}\subseteq \bigcup\{U_i:i\in F_1\}$. Thus $X=\bigcup\{U_i:i\in F_1\}\cup\bigcup\{C_j:j\in F_2\}$. By Alexander subbase theorem, $(X,\tau)$ with the strong topology is compact.
\end{proof}

\begin{proposition}\label{co-noe-well}
	For a $T_0$ Noetherian space $(X,\tau)$, it is a co-Noetherian space if and only if it is open well-filtered.
\end{proposition}
\begin{proof}
	It is well-known in general topology, all sober spaces are well-filtered and all well-filtered spaces are open well-filtered. By Corollary \ref{co-noe-sober}, all $T_0$ co-Noetherian spaces are open well-filtered.
	
	Conversely, since $(X,\tau)$ is Noetherian, we have $\tau\subseteq \mathcal{Q}(X)$(i.e., $U\ll U$ in $(\tau,\subseteq)$, for each $U\in\tau$). The open well-filteredness of $(X,\tau)$ implies that $(X,\tau)$ is co-Noetherian.
\end{proof}

By Theorem \ref{str-compact-1} and Proposition \ref{co-noe-well}, we have the following corollary.

\begin{corollary}\label{str-compact-2}
	Let $(X,\tau)$ be a $T_0$ space. Then the following conditions are equivalent.
	\II
	\I[(1)] $(X,\tau^s)$ is compact.
	\I[(2)] $(X,\tau)$ is both a Noetherian space and a co-Noetherian space.
	\I[(4)] $(X,\tau)$ is both a Noetherian space and a well-filtered space.
	\I[(4)] $(X,\tau)$ is both a Noetherian space and an open well-filtered space.
	\III
\end{corollary}

The following lemma is excerpted from \cite[Lemma 9.7.9]{Goubault-2013}.

\begin{lemma}\label{gou}
	A space $X$ is Noetherian if and only if its sobrification is Noetherian.
\end{lemma}

By Corollary \ref{str-compact-2}, we can directly obtain the following corollary.

\begin{corollary}
	A space $X$ is Noetherian if and only if its well-filterification is Noetherian.
\end{corollary}
\begin{proof}
	If the well-filterification of $X$ is a Noetherian, then it is sober, by Corollary \ref{str-compact-2}. It means that the sobrification is homeomorphic to the well-filterification. So $X$ is Noetherian space by Lemma \ref{gou}.
	
	Conversely, assume that $X$ is a Noetherian space and $X^w$ is the well-filterification space of $X$. Then the sobrification of $X$ and $X^w$ are homeomorphism. And the sobrification of $X$ is Noetherian by Lemma \ref{gou}. Hence, the sobrification of $X^w$ is also Noetherican. It implies that $X^w$ is Noetherian by Lemma \ref{gou}.
\end{proof}

Next, we demonstrate that there is no infinite $T_1$ co-Noetherian space. 

\begin{proposition}\label{t1 co-noe}
	There is no infinite $T_1$ co-Noetherian space.
\end{proposition}
\begin{proof}
	Let $X$ be an infinite set and $(X,\tau)$ be a $T_1$ space. Then $X\not=cl(F)$ for each $F\subseteq_{fin}X$. Thus $(X,\tau)$ is not a co-Noetherian space.
\end{proof}

However, there are two similar properties for $T_1$ spaces and co-Noetherian spaces. 

\begin{proposition}
Let $X,Y$ be two $T_0$ spaces. If $Y$ is a $T_1$ space and there is a continuous injective map from $X$ to $Y$, then $X$ is also a $T_1$ space.
\end{proposition}
\begin{proof}
	Given $x\in X$. Since $f$ is injective, we have $\{x\}=f^{-1}(f(x))$. So $\{x\}$ is closed because $Y$ is $T_1$ and $f$ is continuous.
\end{proof}

\begin{proposition}
	Let $X,Y$ be two $T_0$ spaces. If $X$ is a co-Noetherian space and there is a continuous surjective map from $X$ to $Y$, then $Y$ is also a co-Noetherian space.
\end{proposition}
\begin{proof}
	Fix $C\in\Gamma(Y)$. Then $f^{-1}(C)\in\Gamma(X)$. It follows that $f^{-1}(C)=cl(F)$ for some $F\subseteq_{fin}X$. Assume $U\in \mathcal{O}(Y)$ with $U\cap C\not=\emptyset$. Then there exists $x\in f^{-1}(U\cap C)\subseteq f^{-1}(U)\cap f^{-1}(C)$ due to the fact that $f$ is surjevitve. So $f^{-1}(U)\cap F\not=\emptyset$. That is, $U\cap f(F)\not=\emptyset$ (i.e., $C=cl(f(F))$). It implies that $Y$ is a co-Noetherian space.
\end{proof}

In the following, we study the retractions and some sub-spaces of co-Noetherian spaces. 

\begin{corollary}
	Let $(Y,\nu)$, $(X,\tau)$ be two $T_0$ space such that $(X,\tau)$ is a retraction of $(Y,\nu)$. If $(Y,\nu)$ is a co-Noetherian space, then so is $(X,\tau)$.
\end{corollary}

\begin{proposition}
	Let $(X,\tau)$ be a co-Noetherian space, $A$ be a saturated set and $C$ be a closed set. Then both $A$ and $C$ endowed with hereditary topologies are co-Noetherian.
\end{proposition}
\begin{proof}
	Given a closed subset $B$ of $(A,\tau|_A)$. Then $B=B_0\cap A$ for some $B_0\in\Gamma(X,\tau)$. It follows that $B_0=cl_{X}(F)$ for some $F\subseteq_{fin}X$. Thus $B=cl_{A}(A\cap F)$. 
	
	The proof of closed sub-space is similar.
\end{proof}

The product of co-Noetherian spaces may not be co-Noetherian.

\begin{example}
	Let $L=\{a_i:i\in\mathbb{N}\}\cup\{\infty\}$ and $M=\{b_j:j\in\mathbb{N}\}\cup\{b_0\}$ be two posets, where $\le_L=\{(a_i,a_j):i\le j~in~\mathbb{N}\}\cup\{(a_i,\infty):i\in\mathbb{N}\}$ and $\le_M=\{(b_i,b_j):j\le i~in~\mathbb{N}\cup\{0\}\}$. Then $\Gamma(L,\sigma(L))=\{\downarrow a_i:i\in \mathbb{N}\}\cup\{\downarrow\infty\}$ and $\Gamma(M,\sigma(M))=\{\downarrow b_j:j\in\mathbb{N}\cup\{0\}\}$. Hence, both $(L,\sigma(L))$ and $(M,\sigma(M))$ are co-Noetherian spaces.
	
	Let $C=\{(a_i,b_j):i,j\in\mathbb{N},~i\le j~in~\mathbb{N}\}$. Now we prove that $C$ is a closed subset of $(L\times M,\sigma(L)\times \sigma(M))$.
	
	Clearly, $C$ is a lower set. Take $(a_i,b_j)\not\in C$, where $i\in \mathbb{N}$ and $j\in\mathbb{N}\cup\{0\}$. Then $j<i$ in $\mathbb{N}$ and $\uparrow (a_i,b_j)=\uparrow a_i\times \uparrow b_j\in\sigma(L)\times \sigma(M)$, which implies $\uparrow (a_i,b_j)\cap C=\emptyset$. Similarly, take $(\infty,b_j)\not\in C$ where $j\in\mathbb{N}\cup\{0\}$. Then $(a_{j+1},b_j)\not\in C$ and $(\infty,b_j)\in\uparrow (a_{j+1},b_j)$.
	
	Finally, we claim that there is no finite set $F$ satisfying $C=cl(F)$.
	
	Let $\{(a_{i_k},b_{j_k}):k\in F\}$ be a finite subset of $C$. Take $H=\max\{i_k:k\in F\}$. Then $(a_{H+1},b_{H+1})\in C$, but $(a_{H+1},b_{H+1})\not\in cl(F)$. Thus $(L\times M,\sigma(L)\times \sigma(M))$ is not a co-Noetherian space.
\end{example}

For a sup-complete poset $P$, $(P,\upsilon(P))$ is a strong R-space; however, it need not be co-Noetherian. Furthermore, for a $T_0$ space $X$, the Hare powerspace $P_H(X)$ need not be co-Noetherian.

\begin{example}
 Let $P$ be the set of all at most countably many binary strings endowed with the prefix order, $A$ be the set of all infinite strings and $B$ the set of all finite strings. Then $A=\bigcap\{\downarrow F:F\subseteq_{fin}B\}$ is closed. But $A$ does not have finite maximal elements. Thus, $(P,\upsilon(P))$ is not a co-Noetherian space.
 
 In addition, $P_H(P,\upsilon(P))$ is not a co-Noetherian space since the set of all infinite binary strings is closed without finite maximal elements.
\end{example}

Finally, we prove that the category of co-Noetherian spaces with continuous maps is equivalent to some sub-category of \textbf{DCPO} and discuss some properties of the Smyth powerspace for co-Noetherian spaces. 

\begin{definition}
	Let $L$ be a \textbf{dcpo} and $C$ be an anti-chain of $L$. A finite anti-chain $F\subseteq L$ is called a $cover$ of $C$, if $C\subseteq \downarrow F$ and $C\subseteq\downarrow G$ implies that $F\subseteq\downarrow G$ for each $G\subseteq_{fin}L$.
\end{definition}

\begin{definition}
	A \textbf{dcpo} $L$ is called a $controllable$ \textbf{dcpo}, if $C$ has a cover for each anti-chain $C$.
\end{definition}

\begin{definition}
	Let $L,M$ be two \textbf{dcpo}s. A map $f$ from $L$ to $M$ is called an $upper$-$continuous$ map, if $f$ is monotone and $f^{-1}(\downarrow y)$ has a finite number of maximal points, for each $y\in M$.
\end{definition}

   Now we consider the following categories:
   
   \begin{itemize}
   	\item \textbf{Co-NOE}: the category of all $T_0$ co-Noetherian spaces with continuous maps.
   	\item \textbf{C-DCPO}: the category of all controllable \textbf{dcpo}s with upper-continuous maps.
   \end{itemize}
   
   \begin{theorem}\label{c-dcpo}
   	The categories \textbf{Co-NOE} and \textbf{C-DCPO} are equivalent.
   \end{theorem}
   \begin{proof}
   	Let $L$ be a controllable dcpo. We claim that $(L,\upsilon(L))$ is co-Noetherian. Indeed, we only need to prove that for each anti-chain $C$, the closure of $C$ has a finite number of maximal points. For each closed set $E$ of $(L,\upsilon(L))$, $E=\bigcap\{\downarrow F:E\subseteq \downarrow F,~F\subseteq_{fin}L\}$. Hence,
   	$$cl(C)=\bigcap\{\downarrow F:C\subseteq \downarrow F,~F\subseteq_{fin}L\}.$$
   	Let $F_E$ be the cover of $E$. Then $\downarrow F_E=\bigcap\{\downarrow F:C\subseteq \downarrow F,~F\subseteq_{fin}L\}$ by the definition of cover.
   	Thus $(L,\upsilon(L))$ is co-Noetherian.
   	
   	Let $L,M$ be two controllable dcpos and $f:L\rightarrow M$ be an upper-continuous map. Since $f$ is an upper-continuous map, for each $C\in\Gamma(M,\upsilon(M))$, $f^{-1}(C)=\downarrow F$ for some $F\subseteq_{fin}L$. Hence, $f$ is a morphism in \textbf{Co-NOE}.
   	
   	Conversely, let $(X,\tau)$ be a $T_0$ co-Noetherian space. By sobriety of $(X,\tau)$, $(X,\le_\tau)$ is a dcpo. Given an anti-chain $C$. Then $cl(C)=cl(F)=\downarrow_{\le_\tau} F$ for some anti-chain $F\subseteq_{fin}X$. For each $G\subseteq_{fin}L$, $C\subseteq\downarrow_{\le_\tau} G$ always implies that $cl(C)\subseteq\downarrow_{\le_\tau} G$. Hence, $F\subseteq \downarrow_{\le_\tau} G$. It follows that $F$ is the cover of $C$.
   	
   	Let $(X,\tau),(Y,\nu)$ be two $T_0$ co-Noetherian spaces and $f:X\rightarrow Y$ a continuous map. Clearly, $f$ is monotone from $(X,\le_\tau)$ to $(Y,\le_\nu)$. For each $y\in Y$, $f^{-1}(\downarrow_{\le_\nu} y)\in\Gamma(X)$ and $f^{-1}(\downarrow_{\le_\nu} y)=\downarrow_{\le_\tau} F$ for some $F\subseteq_{fin}X$. Thus $\max(F)$ is the set of maximal points of $f^{-1}(\downarrow_{\le_\nu}y)$.
   	
   	From the proof above, we can define two functors $F,G$ below:
   	
   	$$ F:\textbf{C-DCPO}\rightarrow \textbf{Co-NOE}\\
   	\left\{ 
   	\begin{array}{lc}
   		F(L)=(L,\upsilon(L)), & for~each~object~L~of~\textbf{C-DCPO} \\
   		F(f)=f, & for~each~morphisim~f~of~\textbf{C-DCPO}\\
   	\end{array}
   	\right.$$ 
   	
   	$$ G: \textbf{Co-NOE}\rightarrow \textbf{C-DCPO}\\
   	\left\{ 
   	\begin{array}{lc}
   		F(X)=(X,\le_{\tau}), & for~each~object~(X,\tau)~of~\textbf{Co-NOE} \\
   		G(f)=f, & for~each~morphisim~f~of~\textbf{Co-NOE}\\
   	\end{array}
   	\right.$$ 
   	
   	Thus \textbf{C-DCPO} and \textbf{Co-NOE} are equivalent by $F\circ G=id_{\textbf{C-DCPO}}$ and $G\circ F=id_{\textbf{Co-NOE}}$.
   \end{proof}
   
   \begin{remark}\label{smy-co-noe}
   	For each co-Noetherian space $(X,\tau)$, the topology of $P_S(X,\tau)$ happens to coincide with the upper topology on $(\mathcal{Q}^*(X),\supseteq)$. This follows because the collection $\{\vsqu C:C\in C\}$ itself is a base for the closed sets of $P_S(X)$ and
   	$$\vsqu C=\{K\in\mathcal{Q}^*(X):F\cap K\not=\emptyset\}=\bigcup_{x\in F}\{K\in\mathcal{Q}^*(X):\uparrow x\subseteq Q\}=\downarrow_{P_S(X)}\{\uparrow x:x\in F\}$$ 
   	for each $C=cl(F)=\in\Gamma^*(X)$.
   \end{remark}

   Notably, the Smyth powerspace of a co-Noetherian space may fail to be co-Noetherian. A counterexample will be given by applying Theorem \ref{c-dcpo}. In the following, $\mathbb{N}$ is the set of all positive integers.
   
   \begin{example}
   	Let $P=\mathbb{N}\cup\{\infty\}$ be a poset, where $x\le y$ if and only if $y=\infty$. Then we have the following conditions:
   	
   	(i) $(P,\upsilon(P))$ is a co-Noetherian space since $\Gamma(P,\upsilon(P))=\{\downarrow \infty\}\cup\{\downarrow F:F\subseteq_{fin}\mathbb{N}\}$.
   	
   	(ii) All saturated sets are compact. It is obvious because $\upsilon(P)\subseteq \tau_{cof}(P)$, where $\tau_{cof}(P)$ is the co-finite topology of $P$.
   	
   	(iii) $P_S(X)$ is the poset $(\mathcal{Q}^*(X),\supseteq)$ equipped with upper topology.
   	
   	(iv)  $\mathcal{A}=\{X-\{n\}:n\in\mathbb{N}\}$ is an anti-chain of $P_S(P,\upsilon(P))$ equipped with specialization order, but it does not have a cover. Suppose not, $\mathcal{F}$ is a cover of $\mathcal{A}$. Let $X=(P,\upsilon(P))$. Then there is a $K\in \mathcal{F}$ such that $\mathcal{A}\cap\downarrow_{P_S(X)}K$ is infinite. Given that $x,y\not\in K$ and set $\mathcal{G}=(\mathcal{F}-\{K\})\cup \{K\cup\uparrow x\}\cup \{K\cup\uparrow y\}$. Then $\mathcal{C}\subseteq \downarrow_{P_S(X)} \mathcal{G}$, but $\mathcal{F}\not\subseteq \downarrow_{P_S(X)}\mathcal{G}$, a contradiction. By Theorem \ref{c-dcpo}, $P_S(X)$ is not co-Noetherian.
   \end{example}

\section{Relations of some non-Hausdorff topologies}

In this section, we only consider $T_0$ spaces. From the discussion in the above sections, we obtain a diagram illustrating the classification relations among certain $T_0$ spaces. The implication relations here are not reversible.

	 	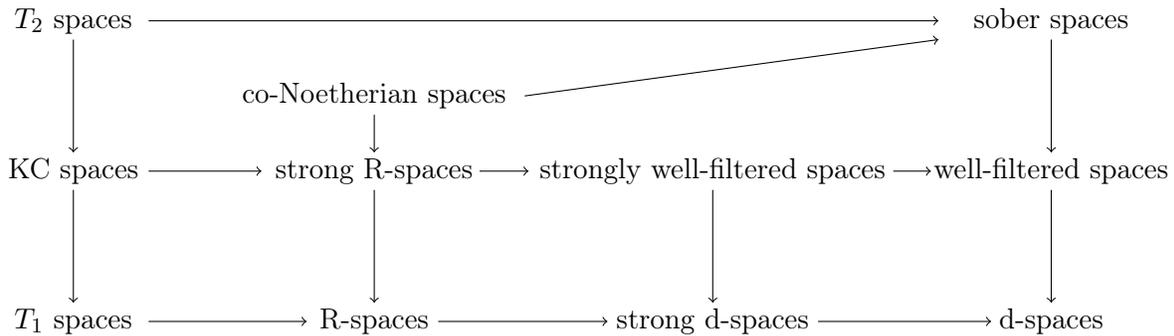
\begin{figure}[H]
	 	\centering
	 	\begin{tikzpicture}[scale=0.5]
	 		\path (10,6) node{sober spaces} coordinate (a);
	 		\path (10,2) node{well-filtered spaces} coordinate (b);
	 		\path (10,-2) node{d-spaces} coordinate (c);
	 		
	 		\path (1,2) node{strongly well-filtered spaces} coordinate (d);
	 		\path (1,-2) node{strong d-spaces} coordinate (e);
	 		
	 		\path (-8,4) node{co-Noetherian spaces} coordinate (f);
	 		\path (-8,2) node{strong R-spaces} coordinate (g);
	 		\path (-8,-2) node{R-spaces} coordinate (h);
	 		
	 		\path (-16,6) node{$T_2$ spaces} coordinate (i);
	 		\path (-16,2) node{KC spaces} coordinate (j);
	 		\path (-16,-2) node{$T_1$ spaces} coordinate (k);
	 		
	 		\draw [->] (-16,5.5) -- (-16,2.5);
	 		\draw [->] (-16,1.5) -- (-16,-1.5);
	 		\draw [->] (-14,6) -- (7,6);
	 		\draw [->] (-14,2) -- (-11,2);
	 		\draw [->] (-5.2,2) -- (-3.9,2);
	 		\draw [->] (5.8,2) -- (6.8,2);
	 		\draw [->] (-14,-2) -- (-9.8,-2);
	 		\draw [->] (-6.3,-2) -- (-1.8,-2);
	 		\draw [->] (3.8,-2) -- (8.4,-2);
	 		\draw [->] (-4,4) -- (7,5.5);
	 		\draw [->] (-8,3.5) -- (-8,2.5);
	 		\draw [->] (-8,1.5) -- (-8,-1.5);
	 		\draw [->] (1,1.5) -- (1,-1.5);
	 		\draw [->] (10,5.5) -- (10,2.5);
	 		\draw [->] (10,1.5) -- (10,-1.5);

	 	\end{tikzpicture}
	 	\caption{Classification Diagram of $T_0$ Spaces }\label{fig3}
	 \end{figure}
	We have added KC-spaces and co-Noetherian spaces to \cite[Figure 1.]{xu-yang-2025} in order to achieve a more refined classification of $T_0$ spaces.
	
	Several issues remain to be solved.
	
	\textbf{Question 5.1.} Is there a $T_1$ coherent well-filtered space but not a KC-space?
	
	\textbf{Question 5.2.} Is there a dcpo $L$ satisfying $\upsilon(L)=\sigma(L)$, but which fails to be controllable?
	
	\textbf{Question 5.3.} Let X be a $T_0$ space and Y a co-Noetherian space. Is the function space $[X\rightarrow Y]$ equipped with Isbell topology co-Noetherian?

	\section*{Reference}
	\bibliographystyle{plain}

\begin{thebibliography}{10}
		\bibitem{abramsky1995domain}
		\newblock S. Abramsky and A. Jung, {\it Domain Theory}, in: S. Abramsky, D. M. Gabbay and T. S. E. Maibaum (eds), Handbook of Logic in Computer Science, Oxford University Press, Oxford, 1995, pp.1--168.
		
		\bibitem{Alas-2002} O. T. Alas, R. G. Wilson, Spaces in which compact subsets are closed and the lattice of $ T_1 $-topologies on a set, Commentationes Mathematicae Universitatis Carolinae, 2002, 43(4): 641-652.
		
		\bibitem{Davey} B. A. Davey, H. A. Priestley,  Introduction to Lattices and Order. Cambridge University Press, Cambridge, 2002.
		
		\bibitem{Gier-2003} G. Gierz, K.H. Hofmann, K. Keimel, J.D. Lawson, M. Mislove, D.S. Scott, Continuous lattices and domains, Encyclopedia of Mathematics and Its Applications, vol. 93, Cambridge university press, 2003.
		
		\bibitem{Goubault-2013} J. Goubault-Larrecq, Non-Hausdorff Topology and Domain Theory, New Mathematical Monographs, vol. 22, Cambridge University Press, 2013.
		
		\bibitem{Hartshorne} R. Hartshorne,  Algebraic Geometry, Springer - Verlag, New York, 1977.
		
		\bibitem{Heck-90} R. Heckmann. Power Domain Constructions, PhD thesis. Universit¨at des Saarlandes, 1990.
		\bibitem{Heck-91} R. Heckmann. An upper power domain construction in terms of strongly compact sets. Lecture Notes
		in Comput. Sci., 598: 272–293, 1991.
		
		\bibitem{Heck-13} R. Heckmann, K. Keimel, Quasicontinuous domains and the Smyth powerdomain, Electron. Notes Theor. Comput. Sci. 298 (2013) 215--232.
		
		
		
		\bibitem{Hochster-1969} M. Hochster, Prime ideal structure in commutative rings. Trans. Amer. Math. Soc., 1969, 142, pp.43-60.
		
		\bibitem{Hoffmann-1979a} R. Hoffmann, On the sobrification remainder $^{s}X-X$, Pac. J. Math. 83 (1979) 145–156. 
		
		\bibitem{Hoffmann-1979b} R. Hoffmann, Sobrification of partially ordered sets, Semigroup Forum 17 (1979) 123–138. 
		
		\bibitem{Johnstone} P. T. Johnstone, Stone Spaces, Cambridge Studies in Advanced Mathematics, Cambridge University Press, Cambridge, 1982.
		
		\bibitem{Keimel-2009} K. Keimel, J. Lawson, D-completions and the d-topology, Ann. Pure Appl. Log. 159 (3), 2009, 292–306.
		
		\bibitem{Xu-2024a} J. Lawson, X. Xu, Spaces determined by countably many locally compact subspaces, Topol. Appl., 2024, 354: 108973.
		
		\bibitem{Xu-2024b} J. Lawson, X. Xu, $T_0$-spaces and the lower topology, Math. Struct. Comput. Sci., 2024: 34(6), 467-490.
		
		\bibitem{Lawson-2020} J. Lawson, G. Wu, X. Xi, Well-filtered spaces, compactness, and the lower topology, Houst. J. Math. 46 (1) (2020) 283–294.
		
		
		
		\bibitem{li-yuan-zhao-2020} Q. Li, Z. Yuan and D. Zhao,
		A unified approach to some non-Hausdorff topological properties, Math. Struct. Comput. Sci. 30 (2020) , 9: 997-1010.
		
		\bibitem{Liu-2020} B. Liu, Q. Li, G. Wu, Well-filterifications of topological spaces, Topol. Appl. 279 (2020) 107245.
		
		\bibitem{lyu-chen-jia-2022} Z. Lyu, Y. Chen, X. Jia, Core-compactness, consonance and the Smyth powerspaces, Topol. Appl. 312(2022): 108066.
		
		\bibitem{Matsumura} H. Matsumura,  Commutative Ring Theory, Cambridge University Press, Cambridge, 1989.
		
		
		\bibitem{Rudin-1980}  M. E. Rudin, Directed sets which converge, Proc. Riverside Symposium on Topology and Modern Analysis, 1980, 305-307.
		
		\bibitem{Shen-2019} C. Shen, X. Xi, X. Xu, D. Zhao, On well-filtered reflections of $T_0$ spaces, Topol. Appl. 267 (2019) 106869.
		
		\bibitem{skula-1969} L. Skula, On a reflective subcategory of the category of all topological spaces, Trans. Amer. Math. Soc., 1969, 142: 37-41.
		
		\bibitem{Wil-1967} A. Wilansky, Between T1 and T2, Amer. Math. Monthly 74 (1967), 261–266.
		
		
		
		
		
		
		
		
		
	
		
		\bibitem{shen-xi-xu-zhao-2020} C. Shen, X. Xi, X. Xu, D. Zhao, On open well-filtered spaces. Log. Methods Comput. Sci. 16(2020)
		
		\bibitem{xi-2017} X. Xi, J. Lawson, On well-filtered spaces and ordered sets, Topol. Appl. 228(2017): 139-144.
		
		\bibitem{xu-2016} X. Xu, Order and Topology, Beijing, Science Press, 2016
		
		\bibitem{xu-2025} X. Xu, Strongly well-filtered spaces and strong d-spaces, Topol. Appl., 2025: 109654.
		
		\bibitem{xu-shen-zhao-2020} X. Xu, C. Shen, X. Xi, D. Zhao, On $T_0$ spaces determined by well-filtered spaces, Topol. Appl., 282(2020): 107323.
		
		\bibitem{xu-yang-2025} X. Xu, Y. Yang, L. Chen, Strong well-filteredness of upper topology on sup-complete posets, arXiv preprint, arXiv:2512.10599, 2025.
		
		\bibitem{xu-zhao-2020}  X. Xu, D. Zhao, 
		On topological Rudin's lemma, well-filtered spaces and sober spaces, Topol. Appl., 272(2020): 107080.
		
		\bibitem{Zhang-2017} Z. Zhang, Q. Li, A direct characterization of the monotone convergence space completion, Topol. Appl. 230 (2017) 99–104.
		
		
		\bibitem{zhao-ho-2015} D. Zhao, W. Ho, On topologies defined by irreducible sets, J. Log. Algebr. Methods Program. 84(1) (2015) 185-195.
		
		
	\end{thebibliography}

\end{document}